\documentclass[twoside,a4paper,12pt,centertags]{amsart}
\usepackage{amsmath,amssymb,verbatim,vmargin}
\usepackage[bookmarks=true]{hyperref}   
\theoremstyle{plain}
\newtheorem{thm}{Theorem}[section]
\newtheorem{lem}[thm]{Lemma}

\newtheorem{prop}[thm]{Proposition}

\theoremstyle{definition}
\newtheorem{defn}[thm]{Definition}
\newtheorem{rem}[thm]{Remark}

\title
{On Hofmann's  bilinear estimate}
\author{Pascal Auscher} 
\address{Pascal Auscher, Universit\'e de Paris-Sud, UMR du CNRS 8628, 91405 Orsay Cedex, France}
\email{pascal.auscher@math.u-psud.fr}

\mathchardef\semic="303B

\newcommand{\R}{{\mathbf R}}
\newcommand{\C}{{\mathbf C}}

\newcommand{\A}{{\mathcal A}}

\newcommand{\mL}{{\mathcal L}}

\newcommand{\mP}{{\mathcal P}}
\newcommand{\mS}{{\mathcal S}}
\newcommand{\mR}{{\mathcal R}}

\DeclareMathOperator{\re}{Re}

\newcommand{\sett}[2]{ \{ #1 \, \semic \, #2 \} }

\newcommand{\nul}{\textsf{N}}
\newcommand{\ran}{\textsf{R}}
\newcommand{\dom}{\textsf{D}}

\newcommand{\conj}[1]{\overline{#1}}

\newcommand{\sgn}{\text{{\rm sgn}}}
\newcommand{\barint}{\mbox{$ave \int$}}
\newcommand{\divv}{{\text{{\rm div}}}}
\newcommand{\curl}{{\text{{\rm curl}}}}

\newcommand{\tb}[1]{\| \hspace{-0.42mm} | #1 \| \hspace{-0.42mm} |}

\newcommand{\wt}{\widetilde}
\newcommand{\ta}{{\scriptscriptstyle \parallel}}

\newcommand{\pd}{\partial}

\newcommand{\oA}{{\overline A}}
\newcommand{\uA}{{\underline A}}

\newcommand{\ol}[1]{{\overline {#1}}}

\def\barint_#1{\mathchoice
            {\mathop{\vrule width 6pt
height 3 pt depth -2.5pt
                    \kern -8.8pt
\intop}\nolimits_{#1}}%
            {\mathop{\vrule width 5pt height
3 pt depth -2.6pt
                    \kern -6.5pt
\intop}\nolimits_{#1}}%
            {\mathop{\vrule width 5pt height
3 pt depth -2.6pt
                    \kern -6pt
\intop}\nolimits_{#1}}%
            {\mathop{\vrule width 5pt height
3 pt depth -2.6pt
          \kern -6pt \intop}\nolimits_{#1}}}

\begin{document}

\begin{abstract}
Using the framework of a previous article joint  with Axelsson and McIntosh, we extend to systems  two results of S. Hofmann for real symmetric equations and their perturbations going back to a work of B. Dahlberg for Laplace's equation on Lipschitz domains, The first one is a certain bilinear estimate for a class of weak solutions  and the second is a criterion which allows to identify the domain of the generator of the semi-group yielding such solutions. \end{abstract}
\maketitle



\section{Introduction}

S. Hofmann proved  in \cite{H} that weak solutions of
\begin{equation}  \label{eq:divformscalar}
  \divv_{t,x} A(x) \nabla_{t,x}U(t,x)= \sum_{i,j=0}^n \pd_i A_{i,j}(x) \pd_j U(t,x) =0
\end{equation}
on the upper half space $\R^{1+n}_+ := \sett{(t,x)\in\R\times \R^n}{t>0}$, $n\ge 1$,
where the matrix $A=(A_{i,j}(x))_{i,j=0}^n\in L_\infty(\R^n;\mL(\C^{1+n}))$ 
is assumed to be $t$-independent and within some small  $L_{\infty}$ neighborhood of  a real symmetric strictly elliptic $t$-independent matrix, obey the following bilinear estimate
$$
\left|\iint_{\R^{1+n}_+} \nabla_{t,x} U \cdot \ol{\bf v} \, dtdx \right| \le C \|U_{0}\|_{2}( \tb{t\nabla {\bf v}} + \|N_{*}{\bf v}\|_{2})
$$
for all $\C^{1+n}$-valued  field ${\bf v}$ such  that the right-hand side is finite. See below for the definition of the square-function $\tb{\ }$ and the non-tangential maximal operator $N_{*}$. 
The trace of $U$ at $t=0$ is assumed to be in the sense of non-tangential convergence a.e. and in $L_{2}(\R^n)$.

In addition, he proves that the solution operator $U_{0}\to U(t,\cdot)$ defines a bounded $C_{0}$ semi-group on $L_{2}(\R^n)$ whose infinitesimal generator $\A$ has domain $W^{1,2}(\R^n)$ with $\|\A f\|_{2}\sim \|\nabla f \|_{2}$. 

Such results  were first proved by  B. Dahlberg \cite{D}  for harmonic functions on a  Lipschitz domain. A version of the bilinear estimate for Clifford-valued monogenic functions was proved by Li-McIntosh-Semmes \cite{LMS}. A short proof of Dahlberg's estimate for harmonic functions and some applications appear in Mitrea's work \cite{M}.  $L^p$ versions are recently discussed  by Varopoulos \cite{Va}. 

Hofmann's arguments for variable coefficients rely on  the deep results of \cite{AAAHK}, and in particular Theorem 1.11 there where  the boundedness and invertibility of the layer potentials are obtained  from a $T(b)$ theorem, Rellich estimates in the case of real symmetric matrices and  perturbation. This also generalizes somehow the case where $A_{0,i}=A_{i,0}=0$ for $i=1, \ldots, n$ corresponding to the Kato square root problem. 

The recent works \cite{AAM,elAAM}, pursuing ideas in  \cite{AAH},  allow us to extend this further to systems, making clear in particular that specificities of real symmetric coefficients and their perturbations  and of equations  - in particular  the De Giorgi-Nash-Moser estimates - are not needed: it only depends on whether the Dirichlet problem is solvable. We use the solution operator constructed in \cite{AAM} and the proof  using $P_{t}-Q_{t}$ techniques of Coifman-Meyer from \cite{CM} makes transparent  the para-product like character of this bilinear estimate. We also establish  a  necessary and sufficient condition telling when the domain of the   infinitesimal generator $\A$ of the Dirichlet semi-group is $W^{1,2}$.  

We apologize to the reader for the necessary conciseness of this note and suggests he (or she) has (at least) the references \cite{AAH, AAM, elAAM} handy. In Section 2, we try to extract from them the relevant information.  The proof or the bilinear estimate for variable coefficients systems is in Section 3. Section 4 contains the discussion on the domain of the Dirichlet semi-group.

\section{Setting}

We begin by giving a precise definition of well-posedness of the  Dirichlet problem for systems.
Throughout this note, we use the notation $X \approx Y$ and $X \lesssim Y$ for estimates
to mean that there exists a constant $C>0$, independent of the variables in the estimate, 
such that $ X/C \le Y \le CX$ and $X\le C Y$, respectively.

We write $(t,x)$ for the standard coordinates  for $\R^{1+n}=\R \times \R^n$,  $t$  standing  for the vertical or normal coordinate. 
For vectors ${\bf v}=({\bf v}_i^\alpha)_{0\le i\le n}^{1\le \alpha\le m} \in \C^{(1+n)m}$, 
we write ${\bf v}_0\in\C^m$ and ${\bf v}_\ta\in\C^{nm}$ for the
normal and tangential parts of ${\bf v}$, 
i.e.~${\bf v}_0= ({\bf v}_0^\alpha)^{1\le \alpha\le m}$ 
whereas
${\bf v}_\ta= ({\bf v}_i^\alpha)_{1\le i\le n}^{1\le \alpha\le m}$. 

For systems, gradient and divergence act as $(\nabla_{t,x}U)_i^\alpha= \pd_i U^\alpha$ and 
$(\divv_{t,x}{\bf F})^\alpha= \sum_{i=0}^n \pd_i {\bf F}^\alpha_i$, with correponding
tangential versions $\nabla_x U= (\nabla_{t,x}U)_\ta$
and $(\divv_x {\bf F})^\alpha= \sum_{i=1}^n \pd_i {\bf F}^\alpha_i$.
With 
 $\curl_{x} {\bf F}_\ta=0$, we understand
$\pd_j{\bf F}_i^\alpha= \pd_i {\bf F}_j^\alpha$, for all $i$, $j=1,\ldots,n,\alpha=1,\ldots, m$.

We consider divergence form second order elliptic systems
\begin{equation}  \label{eq:divform}
  \sum_{i,j=0}^n \sum_{\beta=1}^m \pd_i A^{\alpha,\beta}_{i,j}(x) \pd_j U^\beta(t,x) =0,\qquad \alpha=1,\ldots, m,
\end{equation}
on the half space $\R^{1+n}_+ := \sett{(t,x)\in\R\times \R^n}{t>0}$, $n\ge 1$,
where the matrix $A=(A^{\alpha,\beta}_{ij}(x))_{i,j=0,\ldots ,n}^{\alpha,\beta=1,\ldots ,m}\in L_\infty(\R^n;\mL(\C^{(1+n)m}))$ 
is assumed to be $t$-independent with complex coefficients and {\em strictly accretive} on
$\nul(\curl_\ta)$, in the sense that
there exists $\kappa>0$ such that
\begin{equation}   \label{eq:accrassumption}
  \sum_{i,j=0}^n\sum_{\alpha,\beta=1}^m \int_{\R^n} \re (A_{i,j}^{\alpha,\beta}(x){\bf f}_j^\beta(x) \conj{{\bf f}_i^\alpha(x)}) dx\ge \kappa 
   \sum_{i=0}^n\sum_{\alpha=1}^m \int_{\R^n} |{\bf f}_i^\alpha(x)|^2dx,
\end{equation}
for all 
${\bf f}\in\nul(\curl_\ta):= \sett{{\bf g}\in L_2(\R^n;\C^{(1+n)m})}{\curl_x({\bf g}_\ta)=0}$.  This is nothing but ellipticity in the sense of  G\aa rding. See the discussion in \cite{AAM}.  By changing $m$ to $2m$ we could assume that the coefficients are real-valued. But this does not simplify matters and we need the complex hermitean structure of our $L_{2}$ space anyway. 

\begin{defn}   \label{defn:wellposed}
The Dirichlet problem (Dir-$A$) is said to be {\em well-posed} if for each
$u\in L_2(\R^n;\C^m)$, there is a unique function
$$
U_t(x) =U(t,x)\in C^1(\R_+; L_2(\R^n;\C^m))
$$
such that $\nabla_{x} U\in C^0(\R_+; L_2(\R^n;\C^{nm}))$,
where $U$ satisfies (\ref{eq:divform}) for $t>0$,
$\lim_{t\rightarrow 0}U_t=u$,
$\lim_{t\rightarrow \infty}U_t=0$, $\lim_{t\rightarrow \infty}\nabla_{t,x} U_t=0$ 
in $L_2$ norm,
and 
$\int_{t_0}^{t_1} \nabla_{x} U_s\, ds$
converges in $L_2$ when $t_0\rightarrow 0$ and $t_1\rightarrow\infty$.
More precisely, by $U$ satisfying (\ref{eq:divform}), we mean that 
$\int_t^\infty ((A\nabla_{s,x}U_s)_\ta, \nabla_x v)ds= -( (A\nabla_{t,x}U_t)_0, v)$
for all $v\in C_0^\infty(\R^{n};\C^m)$.
\end{defn}

Restricting to real symmetric equations and their perturbations, this definition is not the one taken in \cite{H} . However, a sufficient condition is provided in \cite{AAM}  to insure that the two methods give rise to the same solution. See also \cite[Corollary 4.28]{AAAHK}. It covers the matrices listed in Theorem  \ref{thm:bvpfordivform} below. This definition is more akin to well-posedness for a Neumann problem (see Section \ref{domain}).

\begin{rem} In the case of block matrices,  ie $A_{0,i}^{\alpha,\beta}(x)=0=A_{i,0}^{\alpha,\beta}(x)$, $1\le i\le n, 1\le \alpha, \beta\le m$, the second order system \eqref{eq:divform} can be solved using semi-group theory: $V(t,\cdot)= e^{-tL^{1/2}}u_{0}$ for $L=-A_{00}^{-1} \divv_{x} A_{\ta\ta} \nabla_{x}$ acting as an unbounded operator on $L_{2}(\R^n, C^{nm})$ (See below for the notation).   This solution satisfies $
V_{t}=V(t,\cdot)\in C^2(\R_+; L_2(\R^n;\C^m)) \cap C^1(\R_{+}, \dom(L^{1/2})),$ $\lim_{t\rightarrow 0}V_{t}=u_{0}$,
$\lim_{t\rightarrow \infty}V_t=0$ 
in $L_2$ norm, and   (\ref{eq:divform}) holds in the strong sense in $\R^n$ for all $t>0$ (and in the sense of distributions in $\R_{+}^{1+n}$).  Hence, the two notions of solvability are not \textit{a priori}  equivalent. That the solutions are the same  follows indeed from the solution of the Kato square root problem for $L$: 
$\dom(L^{1/2})=W^{1,2}(\R^n, C^{nm})$ with $\|L^{1/2}f\|_{2}\sim \|\nabla_{x}f\|_{2}$. See 
\cite{AKMc} where this is explicitly proved when $A_{00}\ne I$.

\end{rem}

The following result is Corollary 3.4 of  \cite{AAM} (which, as we recall, furnishes a different proof of results obtained by combining \cite{JK1} and \cite{DJK} in the case of real symmetric matrices equations ($m=1$)).

\begin{thm}   \label{cor:DJK}
Let $A \in L_\infty(\R^n;\mL(\C^{(1+n)m}))$ be a $t$-independent, complex matrix function
which is strictly accretive on $\nul(\curl_\ta)$ and 
assume that 
(Dir-$A$) is well-posed.
Then any function $U_t(x) =U(t,x)\in C^1(\R_+; L_2(\R^n;\C^m))$ solving (\ref{eq:divform}),
with properties as in Definition~\ref{defn:wellposed}, has estimates
$$
  \int_{\R^n}|u|^2 dx\approx\sup_{t>0}\int_{\R^n}|U_t|^2 dx\approx \int_{\R^n}  | \wt N_*(U)|^2 dx 
  \approx \tb{t\nabla_{t,x} U}^2,
$$
where $u=U|_{\R^n}$.
If furthermore $A$ is real (not necessarily symmetric) and $m=1$, then 
Moser's local boundedness  estimate \cite{Mos} gives the pointwise estimate
$\wt N_*(U)(x) \approx N_*(U)(x)$, where the standard non-tangential maximal function is
$N_*(U)(x):= \sup_{|y-x|<c t}|U(t,y)|$,
for fixed $0<c<\infty$.
\end{thm}

We use the square-function norm 
$$\tb{F_t}^2:= \int_0^\infty \|F_t\|_2^2\, \frac{dt}{t}= \iint_{\R^{1+n}_+} |F(t,x)|^2  \, \frac{dtdx}{t}
$$
and  the following version  $\widetilde N_* (F)$ of the modified {\em non-tangential maximal function} introduced in \cite{KP1}
$$
\widetilde N_*(F)(x):= \sup_{t>0}  t^{-(1+n)/2} \|F\|_{L_2(Q(t,x))},
$$
where $Q(t,x):= [(1-c_0)t,(1+c_0)t]\times B(x;c_1t)$,
for some fixed constants $c_0\in(0,1)$, $c_1>0$.

Next is Theorem 3.2 of \cite{AAM}, specialized to the Dirichlet problem.

\begin{thm}   \label{thm:bvpfordivform}
  The set 
 of matrices $A$ for which (Dir-$A$) is well-posed is an open
  subset of $L_\infty(\R^n; \mL(\C^{(1+n)m}))$.
  Furthermore, it contains
\begin{itemize}
\item[{\rm (i)}] all Hermitean matrices $A(x)= A(x)^*$ (and in particular all real symmetric matrices),
\item[{\rm (ii)}] all block matrices where $A_{0,i}^{\alpha,\beta}(x)=0=A_{i,0}^{\alpha,\beta}(x)$, $1\le i\le n, 1\le \alpha, \beta\le m$, and
\item[{\rm (iii)}] all constant matrices $A(x)=A$.
\end{itemize}
\end{thm}

More importantly is the solution
algorithm using an
``infinitesimal generator'' $T_A$.
Write ${\bf v}\in \C^{(1+n)m}$ as ${\bf v} = [{\bf v}_0, {\bf v}_\ta]^t$, where ${\bf v}_0\in\C^m$ and ${\bf v}_\ta\in\C^{nm}$, and
introduce the auxiliary matrices 
$$
  \oA:=
  \begin{bmatrix}
     A_{00} & A_{0\ta} \\ 0 & I
  \end{bmatrix}, \quad
  \uA:=
  \begin{bmatrix}
     1 & 0 \\ A_{\ta 0} & A_{\ta\ta} 
  \end{bmatrix},
  \qquad\text{if }
  A=
  \begin{bmatrix}
     A_{00} & A_{0\ta} \\ A_{\ta 0} & A_{\ta\ta} 
  \end{bmatrix}
$$
in the normal/tangential splitting of $\C^{(1+n)m}$.
The strict accretivity of  $A$  on $\nul(\curl_\ta)$, as in (\ref{eq:accrassumption}),
implies the pointwise strict accretivity   of the diagonal block $A_{00}$.
Hence $A_{00}$  is invertible, and consequently $\oA$ is invertible [This is not necessarily true for $\uA$.] We define $$T_{A}=\oA^{-1} D \uA$$ as an unbounded operator on $L_{2}(\R^n, \C^{(1+n)m})$ with $D$ the first order self-adjoint operator given in the normal/tangential splitting by
$$
D= \begin{bmatrix}
     0 & \divv_{x} \\ -\nabla_{x} & 0
  \end{bmatrix}.
$$

\begin{prop} \label{prop} Let $A \in L_\infty(\R^n;\mL(\C^{(1+n)m}))$ be a $t$-independent, complex matrix function
which is strictly accretive on $\nul(\curl_\ta)$.  
\begin{enumerate}
\item
 The operator $T_{A}$ has quadratic estimates and a bounded holomorphic functional calculus on $L_{2}(\R^n, \C^{(1+n)m})$. In particular, for any  holomorphic function $\psi$ on the left and right open half planes, with $z\psi(z)$ and $z^{-1}\psi(z)$ qualitatively bounded, one has
$$
\tb{\psi(tT_{A}){\bf f}} \lesssim \|{\bf f}\|_{2}.
$$
\item The Dirichlet problem (Dir-$A$) is well-posed if and only if the operator $$\mS: \ol{\ran(\chi_{+}(T_{A}))} \to  L_{2}(\R^n, \C^{m}), {\bf f} \mapsto {\bf f}_{0}$$ is invertible.  Here, $\chi_{+}=1$ on the right open half plane and 0 on the left  open half plane.
\end{enumerate}
\end{prop}

Item (1) is  \cite[Corollary 3.6]{AAM} (and see \cite{elAAM} for an explicit direct proof)  and item (2) can be found in \cite[Section 4, proof of Theorem 2.2]{AAM}.

\begin{lem}\label{lem:solution} Assume that (Dir-$A$) is well-posed. Let $u_{0}\in L_{2}(\R^n, \C^m)$. Then the solution  $U$ of (Dir-$A$) in the sense of Definition \ref{defn:wellposed} is given by
$$
U(t,\cdot)= (e^{-tT_{A}}{\bf f})_{0}, \quad {\bf f}=\mS^{-1}u_{0} \in \ol{\ran(\chi_{+}(T_{A}))}$$
and furthermore
$$
 \nabla_{t,x}U(t, \cdot)= \partial_{t}  e^{-tT_{A}}{\bf f}.
$$
 \end{lem}

\begin{proof}  \cite[Lemma 4.2]{AAM} (See also \cite[Lemma 2.55]{AAH} with a slightly different  formulation of the Dirichlet problem).
\end{proof}

\section{The bilinear estimate}

We are now in position to state  and prove the generalisation of Hofmann's result. 

\begin{thm}  Assume that (Dir-$A$) is well-posed. Let $u_{0}\in L_{2}(\R^n, \C^m)$ and $U$ be the solution to (Dir-$A$) in the sense of Definition \ref{defn:wellposed}. Then for all ${\bf v}\colon \R^{1+n}_+ \to \C^{(1+n)m}$ such that the right-hand side is finite, 
$$
\left|\iint_{\R^{1+n}_+} \nabla_{t,x} U \cdot \ol{\bf v} \, dtdx \right| \le C \|u_{0}\|_{2}( \tb{t\nabla_{t,x} {\bf v}} + \|N_{*}{\bf v}\|_{2}).
$$

\end{thm}

The pointwise values of ${\bf v}(t,x)$ in the non-tangential control $N_{*}{\bf v}$ can be slightly improved to $L^1$ averages on balls having radii $\sim t$ for each fixed $t$. See the end of proof. 

\begin{proof}  It follows from the previous result  that there exists ${\bf f}\in \ol{\ran(\chi_{+}(T_{A}))}$ such that  $U(t,\cdot)= (e^{-tT_{A}}{\bf f})_{0}$ and 
$$
 \nabla_{t,x}U(t, \cdot)= \partial_{t}{\bf F} = -T_{A}e^{-tT_{A}}{\bf f}, \quad {\bf F}=e^{-tT_{A}}{\bf f}.
$$
Integrating by parts with respect to $t$, we find
$$
\iint_{\R^{1+n}_+} \nabla U \cdot \ol{\bf v} \, dtdx= -\iint_{\R^{1+n}_+} t \partial_{t}{\bf F} \cdot \ol{\partial_{t}{\bf v}} \, dtdx - \iint_{\R^{1+n}_+} t\partial_{t}^2{\bf F} \cdot \ol{\bf v} \, {dtdx}.
$$
The boundary term vanishes because $t\partial_{t}{\bf F}$ goes to $0$ in $L_{2}$ when $t\to 0, \infty$ (this uses  ${\bf f}\in \ol{\ran(\chi_{+}(T_{A}))}$) and $\sup_{t>0}\|{\bf v}(t, \cdot)\|_{2}<\infty$ from $\|N_{*}{\bf v}\|_{2}<\infty$.

For the first term, we use  Cauchy-Schwarz inequality and that $\tb{t\partial_{t}{\bf F}} \lesssim \|u_{0}\|_{2}$ from Theorem \ref{cor:DJK}.

For the second term, we use the following identity: $T_{A}= \oA^{-1} DB
\oA$ with $B=\uA\oA^{-1}$ which, by \cite[Proposition 3.2]{AAM},  is strictly accretive on $\nul(\curl_{\ta})$,  and observe that 
\begin{align*}
t^2\partial_{t}^2{\bf F}&= \oA^{-1}  (tDB)^2e^{-tDB}(\oA {\bf f})
\\
 &= \oA^{-1} (tDB)(I+(tDB)^2)^{-1} \psi(tDB) (\oA {\bf f})
 \\
&= \oA^{-1} (tDB)(I+(tDB)^2)^{-1} \oA \psi(tT_{A})({\bf f})
\end{align*}
with 
$$\psi(z)= z(1+z^2)e^{-(\sgn{\re z})z}.
$$
Thus,
$$
\iint_{\R^{1+n}_+} t\partial_{t}^2{\bf F} \cdot \ol{\bf v} \, {dtdx}= 
\iint_{\R^{1+n}_+} \oA \psi(tT_{A}) ({\bf f})  \cdot  \ol{Q_{t}{\bf v}_{t}} \, \frac{dtdx}t
$$
with $Q_{t}=\Theta_{t}{\oA^{-1}}^*$ and $\Theta_{t}= (tB^*D) (I+(tB^*D)^{2})^{-1}
$
acting on $ {\bf v}_{t}\equiv {\bf v(t,\cdot})$ for each fixed $t$ [The notation $\oA$ has nothing to do with complex conjugate and we apologize for any conflict this may cause.]
It follows from the quadratic estimates of Proposition \ref{prop} 
 that
$$
\tb{\psi(tT_{A}) ({\bf f})} \lesssim  \|{\bf f}\|_{2}.
$$
It remains to estimate $\tb{Q_{t}{\bf v}_{t}}$. To do that we follow the principal part approximation of \cite{elAAM} - which is an elaboration of the so-called Coifman-Meyer trick \cite{CM} - applied to $Q_{t}$ instead of $\Theta_{t}$ there. That is,  we write
\begin{equation}\label{dec}
Q_{t}{\bf v}_{t}=Q_{t}\left(\frac{I-P_{t}}{t(-\Delta)^{1/2}}\right) t(-\Delta)^{-1/2}{\bf v}_{t} + (Q_{t}P_{t}-\gamma_{t}S_{t}P_{t}){\bf v}_{t} + \gamma_{t}S_{t}P_{t}{\bf v}_{t}
\end{equation}
where $\Delta
$ is the Laplacian on $\R^n$, $P_{t}$ is a nice scalar approximation to the identity acting componentwise on $L_{2}(\R^n, \C^{(1+n)m})$ and  $\gamma_{t}$ is the element  
of $ L^2_{\text{loc}}(\R^n; \mL(\C^{(1+n)m}))$ given by $$
   \gamma_t(x){\bf w}:= (Q_t {\bf w})(x)
$$
for every ${\bf w}\in \C^{(1+n)m}$. We view ${\bf w}$ on the right-hand side
of the above equation as the constant function valued in $\C^{(1+n)m}$ defined on $\R^n$ 
by ${\bf w}(x):={\bf w}$. 
We identify $\gamma_t(x)$ with the (possibly unbounded) multiplication
operator $\gamma_t: f(x)\mapsto \gamma_t(x)f(x)$. Finally, the {\em dyadic averaging operator} $S_t:
L_{2}(\R^n, \C^{(1+n)m}) \rightarrow L_{2}(\R^n, \C^{(1+n)m})$ is given by 
$$
  S_t {\bf u}(x) :=   \frac{1}{|Q|} \int_Q
  {\bf u}(y)\,  dy
$$
for every $x \in \R^n$ and $t>0$, where  $Q $ is
the unique dyadic cube  in $ \R^n$ that contains $x$ and has side length $\ell$ with $\ell/2 < t \le \ell $. 

With this in hand, we apply the triple bar norm to \eqref{dec}. 

Using the uniform $L_{2}$ boundedness of $Q_{t}$ and that of $\frac{1-P_{t}}{t(-\Delta)^{1/2}}$, the first term in the RHS  is bounded by $\tb{t(-\Delta)^{1/2}{\bf v}_{t} }\le \tb{t\nabla_{x} {\bf v}_{t}}$.

Following exactly the computation of Lemma 3.6 in \cite{elAAM}, the second term in the RHS is bounded by
$C\tb{t\nabla_{x} P_{t} {\bf v}_{t}} \le   C\tb{t\nabla_{x} {\bf v}_{ t}}$ using the uniform $L_{2}$ boundedness of $P_{t}$. This computation makes use of the off-diagonal estimates of $\Theta_{t}$, hence of $Q_{t}$,  proved in \cite[Proposition 3.11]{elAAM}.

For the third term in the RHS, we observe that $\gamma_{t}(x){\bf w}  =\Theta_{t}({\oA^{-1}}^*{\bf w})(x)$. Hence, the square-function estimate on $\Theta_{t}$ proved in \cite[Theorem 1.1]{elAAM}, the  off-diagonal estimates of $\Theta_{t}$  and the fact that $\oA^{-1}$ is bounded imply that $|\gamma_{t}(x)|^2\frac{dtdx}{t}$ is a Carleson measure. 
 Hence, from Carleson embedding theorem  the third term contributes  $\|N_{*}(S_{t}P_{t}{\bf v})\|_{2}$, which is controlled pointwise by the non-tangential maximal function in the statement with appropriate opening.
\end{proof}

\section{The domain of the Dirichlet semi-group}\label{domain}

Assume (Dir-$A$)  in the sense of Definition \ref{defn:wellposed}  is well-posed.  If we set
$$ \mP_{t}u_{0}=(e^{-tT_{A}}{\bf f})_{0}, \quad {\bf f}=\mS^{-1}u_{0} \in \ol{\ran(\chi_{+}(T_{A}))}$$ for all $t>0$, then Lemma \ref{lem:solution} implies that $(\mP_{t})_{t>0}$ is a bounded $C_{0}$-semigroup on $L_{2}(\R^n, \C^m)$ [Recall that well-posedness includes uniqueness and this allows to prove the semigroup property].

Furthermore, with our definition of well-posedness of  the Dirichlet problem, the domain of the infinitesimal generator $\A$ of this semi-group is  contained in the Sobolev space $W^{1,2}(\R^n, \C^m)$ and $ \|\nabla_{x }u_{0}\|_{2} \lesssim \|\A u_{0}\|_{2}$. Indeed, from  Lemma \ref{lem:solution}  we have for all $t>0$, 
$
\partial_{t}  e^{-tT_{A}}{\bf f}= \nabla_{t,x}U(t, \cdot).
$
Also $\partial_{t}  e^{-tT_{A}}{\bf f} \in \ol{\ran(\chi_{+}(T_{A}))}$ and the invertibility of $\mS$ tells that $\nabla_{t,x}U(t, \cdot) = \mS^{-1}(\partial_{t}U(t,\cdot))$. Therefore
$$
 \|\nabla_{x}U(t, \cdot))\|_{2} \lesssim \|\partial_{t}U(t,\cdot)\|_{2}.$$  By definition of $\A$, $\partial_{t}U(t,\cdot) = \A U(t,\cdot)$, thus we have for all $t>0$
$$
 \|\nabla_{x}U(t, \cdot))\|_{2}\lesssim \|\A U(t,\cdot)\|_{2} .$$
The conclusion for the domain follows easily. 

The question of whether this domain coincides with  $W^{1,2}(\R^n, \C^m)$
is answered by the following theorem

\begin{thm} Assume that (Dir-$A$) and (Dir-$A^*$) are well-posed. Then the domain of the infinitesimal generator $\A$ of $(\mP_{t})_{t>0}$ coincides with the Sobolev space $W^{1,2}(\R^n, \C^m)$ and $ \|\nabla_{x }u_{0}\|_{2} \sim \|\A u_{0}\|_{2}$. 
\end{thm}

This theorem applies to the three situations listed in Theorem \ref{thm:bvpfordivform}.

\begin{proof} Combining \cite[Lemma 4.2]{elAAM} (which says that (Dir-$A^*$) is equivalent to an auxiliary Neumann problem for $A^*$), \cite[Proposition 2.52]{AAH} (which says that this auxiliary Neumann problem is equivalent to a regularity problem for $A$: this is non trivial) with the proof ot Theorem 2.2 in \cite{elAAM} (giving the necessary and sufficient condition below for well-posedness of the regularity problem for $A$), we have that (Dir-$A^*$) is well-posed if and only if 
$$\mR: \ol{\ran(\chi_{+}(T_{A}))} \to  L_{2}(\R^n, \C^{nm}), {\bf f} \mapsto {\bf f}_{\ta}$$ is invertible. This implies that for ${\bf f} \in \ol{\ran(\chi_{+}(T_{A}))}$, we have that 
$$\|{\bf f}\|_{2} \sim \|{\bf f}_{\ta}\|_{2}.$$
Therefore, the conjunction of well-posedness for (Dir-$A$) and (Dir-$A^*$) gives 
 $$\|{\bf f}_{0}\|_{2}   \sim \|{\bf f}_{\ta}\|_{2}, \quad {\bf f} \in \ol{\ran(\chi_{+}(T_{A}))}.$$
 From this, it is easy to identify the domain of $\A$ by  an argument as before.
 \end{proof}
 
We have seen that  invertibility of $\mS$ reduces to  that of  $\mR$ (up to taking adjoints). The only known  way to prove it in such a generality (except for constant coefficients) is via a continuity method  and the Rellich estimates showing that $\|{\bf f}_{\ta}\|_{2} \sim \|(A{\bf f})_{0}\|_{2}$ for all ${\bf f} \in \ol{\ran(\chi_{+}(T_{A}))}.$ This method was first used in the context of Laplace equation on Lipschitz domains by Verchota \cite{V}.  This depends strongly of $A$. Various relations between Dirichlet, regularity and Neumann problems for $L^p$ data in the sense of non tangential approach for second order real symmetric equations are studied in \cite{KP1,KP2} and more recently in \cite{KS,S}.

\bibliographystyle{acm}

\end{document}